\documentclass[11pt,reqno]{amsart}

\usepackage[T1]{fontenc}
\usepackage{lmodern}
\usepackage[a4paper,margin=1in]{geometry}
\usepackage{amsmath,amssymb,amsthm,mathtools}
\usepackage{microtype}
\usepackage[hidelinks]{hyperref}

\numberwithin{equation}{section}

\newtheorem{theorem}{Theorem}[section]
\newtheorem{proposition}[theorem]{Proposition}
\newtheorem{lemma}[theorem]{Lemma}
\newtheorem{corollary}[theorem]{Corollary}

\theoremstyle{remark}
\newtheorem{remark}[theorem]{Remark}

\newcommand{\D}{\mathbb D}
\newcommand{\T}{\mathbb T}

\newcommand{\R}{\mathbb R}

\newcommand{\Hh}{\mathbb H}

\newcommand{\Def}{\operatorname{Def}}
\newcommand{\calM}{\mathcal M}
\newcommand{\calI}{\mathcal I}
\newcommand{\calG}{\mathcal G}
\newcommand{\calA}{\mathcal A}
\newcommand{\Gap}{\mathsf G}

\newcommand{\red}{\widehat R}
\newcommand{\ee}[1]{e^{2\pi i #1}}

\title[Endpoint counterexamples]{Beurling--Carleson Endpoint Counterexamples:
Singular Inner Functions and a Semilinear Equation}

\author[P.~C.~Fang]{Pengcheng Fang}
\address{School of Mathematics, Shanghai University of Finance and Economics, Shanghai 200433, P. R. China}
\email{fpc426853@gmail.com}

\author[Y.~He]{Yixin He}
\address{School of Mathematical Sciences, Fudan University, Shanghai 200433, P. R. China}
\email{yixin.he717@gmail.com}

\subjclass[2020]{Primary 30H10, 35J91; Secondary 30J15, 31A20, 31C35, 31C40}
\keywords{Beurling--Carleson set, singular inner function, Hardy space,
nearly maximal solution, deficiency measure}

\begin{document}
\begin{abstract}
We settle two endpoint problems for Beurling--Carleson support
conditions. For \(0<s<\frac12\) and
\(\theta=\frac{1-2s}{1-s}\), we construct a nonatomic singular
probability measure \(\mu\), supported on a single
\(\theta\)-Beurling--Carleson set, such that
\(S_\nu'\notin H^s\) for every nonzero submeasure
\(0<\nu\leq\mu\). For \(m>3\) and
\(\alpha=\frac{m-3}{m-1}\), we construct a nonatomic probability
measure supported on a single \(\alpha\)-Beurling--Carleson set that is
not the deficiency measure of any nearly maximal solution of
\(\Delta u=(u_+)^m\). Thus hereditary failure persists at the first
endpoint, while the critical support condition in the second problem
is necessary but not sufficient. The proofs combine endpoint
Cantor--Moran constructions with kernel divergence and a nonlinear
energy obstruction.
\end{abstract}

\maketitle

\section{Introduction}
Let $0<\eta<1$.  A closed null set $E\subset\T$ is called an
$\eta$-Beurling--Carleson set if
\[
\sum_{J\subset\T\setminus E}|J|^\eta<\infty,
\]
where the sum is over the complementary arcs of $E$.  This power-gauge
condition belongs to the exceptional-set theory originating with Beurling and
Carleson \cite{Beu40,Car52}.  We write $\calM_\eta(\T)$ for the
finite positive measures concentrated on countable unions of such sets.  The
purpose of this paper is to settle two endpoint questions from
\cite{IN24} in which this boundary geometry occurs.

The two questions share the same critical defect.  Away from the endpoint, a
gain in the power of $|J|$ absorbs the distribution of mass; at the endpoint, the
geometric support sum may stay finite while accumulation over scales diverges.
We treat the analytic and elliptic problems together to isolate this separation.

We first consider the analytic problem.  For a finite positive singular
measure $\mu$ on $\T$, let
\[
S_\mu(z)=\exp\!\left(-\int_{\T}
\frac{\xi+z}{\xi-z}\,d\mu(\xi)\right)
\]
be the associated singular inner function.  The relation between the boundary concentration of $\mu$ and the Hardy regularity of $S_\mu'$ goes back to \cite{AC74,Ahe79,Cul71}; see also \cite{Ivrii19,Mas13}.

Fix $0<s<\tfrac12$ and put
\[
q_c=\frac{s}{1-s},
\qquad
\theta=1-q_c=\frac{1-2s}{1-s}.
\]
Ivrii and Nicolau proved that the condition
$\sum_J|J|^{1-q}<\infty$ implies $S_\sigma'\in H^s$ for every singular
measure supported on $E$ when $q>q_c$.  They also constructed one bad measure
at $q=q_c$, and a measure with hereditary failure for every $q<q_c$; see
\cite[Lemmas~5.4--5.5]{IN24}.  The endpoint example for one measure does not
formally imply hereditary failure, since removing mass may restore the
integrability of the positive boundary kernel.  The first result below shows
that hereditary failure nevertheless persists at the critical exponent.

\begin{theorem}\label{thm:main}
Let $0<s<\tfrac12$ and set
\(\theta=(1-2s)/(1-s)\).
Then there exist a closed set $E\subset\T$ of Lebesgue measure zero
and a nonatomic singular probability measure $\mu$ supported on $E$
such that
\[
\sum_J |J|^\theta<\infty,
\]
where the sum is taken over all complementary arcs $J$ of $E$, and
for every nonzero finite positive Borel measure $\nu$ satisfying
\(0<\nu\le\mu\),
\[
S_\nu'\notin H^s.
\]
\end{theorem}

Here $\nu\le\mu$ denotes the order relation on positive measures.  Thus the
endpoint obstruction is hereditary under passage to nonzero submeasures.
Corollary~\ref{cor:logintro} gives a logarithmic strengthening of the support
condition.

We next turn to the elliptic problem.  For $m>3$, consider
\begin{equation}\label{eq:intro-semilinear}
\Delta u=(u_+)^m\qquad\text{in }\D,
\end{equation}
and let $U$ denote its maximal solution.
The Keller--Osserman theory and its refinements give
\[
U(z)\asymp(1-|z|)^{-2/(m-1)};
\]
see \cite{Kel57,Oss57,BM92,BM98}.  Following \cite{IN24}, a solution $u$ is
nearly maximal if
\[
\limsup_{r\to1^-}\int_{\T}(U-u)(r\xi)\,|d\xi|<\infty.
\]
The weak boundary trace of $U-u$ is its deficiency measure, and $\Def_m$
denotes the class of all such measures.  This is a finite measure trace after
the universal leading blow-up has been removed; compare the broader measure
framework in \cite{Ponce16}.

Set \(\alpha=(m-3)/(m-1)\).
Ivrii and Nicolau established
\begin{equation}\label{eq:intro-IN-inclusions}
\Def_m\subset\calM_\alpha(\T),
\qquad
\calM_\beta(\T)\subset\Def_m
\quad(\beta<\alpha),
\end{equation}
leaving the equality case open \cite[Theorem~1.3]{IN24}.

\begin{theorem}\label{thm:main-counterexample}
Let $m>3$ and set
\(\alpha=(m-3)/(m-1)\).
Then the inclusion
\(\Def_m\subset\calM_\alpha(\T)\)
is strict.  More precisely, there exist a closed null set
$E\subset\T$ and a nonatomic probability measure $\mu$
supported on $E$ such that
\[
\sum_{J\subset\T\setminus E}|J|^\alpha<\infty,
\]
but $\mu\notin\Def_m$.
\end{theorem}

The measure is carried by one $\alpha$-Beurling--Carleson set; thus the
obstruction lies in its mass distribution, not in the countable-union clause in
$\calM_\alpha(\T)$.

\subsection*{Ideas of the proofs and organization}
For Theorem~\ref{thm:main}, the boundary formula for $S_\nu'$ reduces the
problem to a positive kernel.  A binary construction with cylinder length
$L_n$ is chosen so that
\[
\sum_n2^nL_n^\theta<\infty,
\qquad
2^{n(1-s)}L_n^{1-2s}\asymp\frac1{n+1}.
\]
The first relation gives the critical support condition and the second gives a
harmonic-series lower bound for every $0<\nu\le\mu$.

For Theorem~\ref{thm:main-counterexample}, writing $v=U-u$ gives a
Poisson--Green representation and the necessary estimate
\[
v=P\mu-\calG\lambda,
\qquad
\int_{\D}\delta(z)\calA(z,v(z))\,dA(z)<\infty.
\]
Relative Fatou theory and minimal thinness
\cite{Doo59,Doob84,AG01,Bur86,GM05} locate Whitney-scale points where the
Poisson term dominates.  A uniform reduction estimate and Harnack chains
\cite{GT01} propagate this lower bound to disjoint boxes.  The endpoint Moran
parameters satisfy
\[
\sum_n a_nk_n^\alpha<\infty,
\qquad
\sum_n a_nk_n=\infty,
\]
which separates the support condition from the nonlinear energy condition.

Sections~\ref{sec:boundary}--\ref{sec:log} prove the analytic result and its
logarithmic refinement.  Sections~\ref{sec:preliminaries}--
\ref{sec:proof-counterexample} establish the elliptic counterexample.

\subsection*{Notation and conventions}
We identify \(\T\) with \(\R/\mathbb Z\) by
\(x\mapsto e^{2\pi i x}\).  Arc lengths and the measure \(dx\) on
\(\R/\mathbb Z\) are normalized; under this parametrization, \(|d\xi|=2\pi\,dx\).
In the PDE part we use the trace
normalization of \cite{IN24}, with the unnormalized boundary element
\(|d\xi|\) and the factor \(1/(2\pi)\) in the Poisson integral.  The symbol
\(dA\) denotes planar Lebesgue area measure in \(\D\), and
\(t_+=\max\{t,0\}\).  We write \(A\lesssim B\) if \(A\le CB\), with \(C\)
independent of the scale parameters under consideration, and
\(A\asymp B\) if \(A\lesssim B\lesssim A\).

\section{A boundary formula}\label{sec:boundary}
In Sections~\ref{sec:boundary}--\ref{sec:log} we fix
\(0<s<\tfrac12\), and set
\(q_c=s/(1-s)\) and
\(\theta=1-q_c=(1-2s)/(1-s)\).
For a finite positive singular measure \(\sigma\) on \(\T\), put
\[
S_\sigma(z)=\exp\!\left(-\int_{\T}
\frac{\ee{x}+z}{\ee{x}-z}\,d\sigma(x)\right),
\qquad z\in\D.
\]
We shall repeatedly use
\begin{equation}\label{eq:elementary-identities}
\frac{\theta}{1-2s}=\frac{1}{1-s},
\qquad
\frac{1-2s}{\theta}=1-s.
\end{equation}

\begin{lemma}[Ahern--Clark formula]\label{lem:boundary}
Let \(\sigma\) be a finite positive singular measure supported on a closed
null set \(F\subset\T\).  Then, for every
\(x\in\T\setminus F\),
\[
|S_\sigma'(\ee{x})|
=2\int_F\frac{d\sigma(y)}{|\ee{y}-\ee{x}|^2}.
\]
If the function on the right does not belong to
\(L^s(\T\setminus F)\), then \(S_\sigma'\notin H^s\).
\end{lemma}

\begin{proof}
The identity is the Ahern--Clark formula in the singular-inner-function case;
see \cite{AC74}, \cite[Section~4.1]{Mas13}, and \cite[p.~2597]{IN24}.  If
\(S_\sigma'\in H^s\), its radial boundary function belongs to
\(L^s(\T)\) by the standard boundary theorem
\cite[Chapter~2, \S2.1]{Dur70}.  Since \(F\) is null, the asserted
contrapositive follows.
\end{proof}

\section{The endpoint construction}\label{sec:endpoint}
We use a binary construction modeled on
\cite[Lemmas~5.4--5.5]{IN24}, with endpoint scales chosen so that every
nonzero submeasure $0<\nu\le\mu$ inherits the kernel divergence.

Choose a closed arc $I_0\subset\T$ of length $L_0\in(0,\tfrac14)$, and fix a compact
lift of $I_0$ to an interval of $\R$. Define
\begin{equation}\label{eq:Ln-def}
L_n=L_0\,2^{-n/\theta}(n+1)^{-1/(1-2s)},
\qquad n\ge 0.
\end{equation}
Since $0<\theta<1$,
\[
\frac{L_{n+1}}{L_n}=2^{-1/\theta}\Bigl(\frac{n+1}{n+2}\Bigr)^{1/(1-2s)}
<2^{-1/\theta}<\frac12,
\qquad n\ge0.
\]

Starting from $I_0$, we perform the usual binary Cantor construction. At stage $n$,
each basic interval has length $L_n$; from it we keep the two closed endpoint
subintervals of length $L_{n+1}$ and delete the open middle interval. Let $E_n$ be
the union of the $2^n$ remaining intervals at stage $n$, and set
\[
E=\bigcap_{n\ge0} E_n.
\]
For each $n\ge0$ we enumerate the generation-$n$ basic intervals as
\(I_{n,1},\dots,I_{n,2^n}\),
and we write $\Gap_{n+1,k}$ for the open middle interval removed from $I_{n,k}$ at the
next step. Thus \(|I_{n,k}|=L_n\) and
\(|\Gap_{n+1,k}|=g_{n+1}:=L_n-2L_{n+1}\).
Since $L_{n+1}/L_n\le 2^{-1/\theta}$, there exists a constant
\(c_{\mathrm{gap}}=c_{\mathrm{gap}}(s)>0\) such
that
\begin{equation}\label{eq:gap-comp}
c_{\mathrm{gap}}L_n\le g_{n+1}\le L_n,
\qquad n\ge0.
\end{equation}

\begin{proposition}\label{prop:set}
The set $E$ is closed, has Lebesgue measure zero, satisfies
\[
\sum_J|J|^\theta<\infty,
\]
and supports a nonatomic singular probability measure $\mu$ such that
\[
\mu(I_{n,k})=2^{-n}
\qquad(n\ge0,\ 1\le k\le2^n).
\]
\end{proposition}

\begin{proof}
The set \(E\) is compact and
\[
|E_n|=2^nL_n
=L_0\,2^{n(1-1/\theta)}(n+1)^{-1/(1-2s)}\longrightarrow0,
\]
so \(|E|=0\).  Its complementary arcs are the outer arc
\(\T\setminus I_0\) and the deleted middle intervals.  Hence, by
\eqref{eq:gap-comp} and \eqref{eq:elementary-identities},
\[
\sum_J|J|^\theta
\le (1-L_0)^\theta+C\sum_{n\ge0}2^nL_n^\theta
=(1-L_0)^\theta+CL_0^\theta
\sum_{n\ge0}(n+1)^{-1/(1-s)}<\infty.
\]

Assign mass $2^{-n}$ to every generation-$n$ cylinder.  These masses are
consistent and define the usual Bernoulli probability measure $\mu$ on $E$.
Since the largest cylinder mass is $2^{-n}\to0$, the measure is nonatomic;
because $|E|=0$, it is singular.
\end{proof}

\begin{proposition}\label{prop:kernel}
Let $\nu$ be a finite positive Borel measure on $\T$ satisfying $0<\nu\le\mu$, and
set
\[
K_\nu(x)=\int_E \frac{d\nu(y)}{|\ee{x}-\ee{y}|^2},
\qquad x\in\T\setminus E.
\]
Then
\[
\int_{\T\setminus E} K_\nu(x)^s\,dx=\infty.
\]
\end{proposition}

\begin{proof}
Since $\nu\le\mu$ and $\mu(\T\setminus E)=0$, we have $\nu(\T\setminus E)=0$. Hence
$\nu$ is supported on $E$, and \(\nu(E)=\nu(\T)>0\).

Fix $n\ge0$ and $1\le k\le 2^n$. If $x\in \Gap_{n+1,k}$ and $y\in E\cap I_{n,k}$, then
$x$ and $y$ lie in the same lifted copy of $I_{n,k}$, so
\[
|x-y|\le |I_{n,k}|=L_n.
\]
Since $L_0<\tfrac14$, we also have $|x-y|<\tfrac12$, and therefore
\[
|\ee{x}-\ee{y}|=2|\sin \pi(x-y)|\le 2\pi |x-y|\le 2\pi L_n.
\]
It follows that, for every $x\in \Gap_{n+1,k}$,
\[
K_\nu(x)
\ge \int_{E\cap I_{n,k}}\frac{d\nu(y)}{|\ee{x}-\ee{y}|^2}
\ge \frac{\nu(I_{n,k})}{(2\pi L_n)^2}.
\]
Using also \eqref{eq:gap-comp}, we obtain
\begin{align}
\int_{\Gap_{n+1,k}} K_\nu(x)^s\,dx
&\ge g_{n+1}\left(\frac{\nu(I_{n,k})}{(2\pi L_n)^2}\right)^s \notag\\
&\ge c\,\nu(I_{n,k})^sL_n^{1-2s},
\label{eq:one-gap}
\end{align}
where $c>0$ depends only on $s$.

Summing \eqref{eq:one-gap} over $k$ yields
\begin{equation}\label{eq:level}
\int_{\bigcup_{k=1}^{2^n} \Gap_{n+1,k}} K_\nu(x)^s\,dx
\ge cL_n^{1-2s}\sum_{k=1}^{2^n}\nu(I_{n,k})^s.
\end{equation}
Now $\nu(I_{n,k})\le\mu(I_{n,k})=2^{-n}$ for every $k$. Since $0<s<1$, for
$0\le a\le A$ one has $a^s\ge A^{s-1}a$. Applying this with $A=2^{-n}$ gives
\[
\nu(I_{n,k})^s\ge 2^{n(1-s)}\nu(I_{n,k}).
\]
Because $\nu$ is supported on $E\subset E_n=\bigcup_{k=1}^{2^n} I_{n,k}$, we have
\[
\sum_{k=1}^{2^n}\nu(I_{n,k})=\nu(E_n)=\nu(E)=\nu(\T).
\]
Hence
\begin{equation}\label{eq:lp-lower}
\sum_{k=1}^{2^n}\nu(I_{n,k})^s\ge \nu(\T)\,2^{n(1-s)}.
\end{equation}
Combining \eqref{eq:level} and \eqref{eq:lp-lower}, we obtain
\[
\int_{\bigcup_{k=1}^{2^n} \Gap_{n+1,k}} K_\nu(x)^s\,dx
\ge c\,\nu(\T)\,2^{n(1-s)}L_n^{1-2s}.
\]
Finally, by \eqref{eq:Ln-def} and \eqref{eq:elementary-identities},
\[
2^{n(1-s)}L_n^{1-2s}
=L_0^{1-2s}
2^{n(1-s)-n(1-2s)/\theta}
(n+1)^{-1}
=\frac{L_0^{1-2s}}{n+1}.
\]
Therefore
\[
\int_{\bigcup_{k=1}^{2^n} \Gap_{n+1,k}} K_\nu(x)^s\,dx
\ge \frac{c\,\nu(\T)L_0^{1-2s}}{n+1}.
\]
The sets $\bigcup_{k=1}^{2^n} \Gap_{n+1,k}$ are pairwise disjoint as $n$ varies, so
summing over $n$ gives
\[
\int_{\T\setminus E} K_\nu(x)^s\,dx
\ge c\,\nu(\T)L_0^{1-2s}\sum_{n=0}^{\infty}\frac{1}{n+1}=\infty.
\]
This proves the proposition.
\end{proof}

\begin{proof}[Proof of Theorem~\ref{thm:main}]
Apply the construction above.  If $0<\nu\le\mu$, then $\nu$ is a
nonzero singular measure supported on $E$.  By Proposition~\ref{prop:kernel} and
Lemma~\ref{lem:boundary},
\[
|S_\nu'(\ee{x})|=2K_\nu(x),
\qquad x\in\T\setminus E.
\]
Since \(K_\nu\notin L^s(\T\setminus E)\), it follows that
\(S_\nu'\notin H^s\).  The remaining assertions follow from
Proposition~\ref{prop:set}.
\end{proof}

\begin{remark}
In the notation of \cite{IN24}, Theorem~\ref{thm:main} is the endpoint
analogue of the hereditary counterexample from \cite[Lemma~5.5]{IN24}.
Ivrii and Nicolau already gave an endpoint non-hereditary example in
\cite[Lemma~5.4]{IN24}; the point here is that the stronger hereditary
failure $S_\nu'\notin H^s$ for every nonzero $0<\nu\le\mu$ still occurs at the
critical exponent $q_c=s/(1-s)$.
\end{remark}

\section{A logarithmic refinement}\label{sec:log}	
The same construction yields the following logarithmic refinement.
\begin{corollary}\label{cor:logintro}
Let $0<s<\tfrac12$, and let $q_c$ and $\theta$ be as above. Then there exist a
closed set $E\subset\T$ of Lebesgue measure zero and a nonatomic singular
probability measure $\mu$ supported on $E$ such that
\[
\sum_J |J|^{\theta}\Bigl(\log\frac{e}{|J|}\Bigr)^{q_c}<\infty,
\]
where the sum is taken over all complementary arcs \(J\) of \(E\), and for
every nonzero finite positive Borel measure $\nu$ with $0<\nu\le\mu$,
\[
S_\nu'\notin H^s.
\]
\end{corollary}

\begin{proof}[Proof of Corollary~\ref{cor:logintro}]
Define
\[
A_n=\frac1{(n+2)\log(n+2)},
\qquad
\widetilde L_n
=L_0\,2^{-n/\theta}\left(\frac{A_n}{A_0}\right)^{1/(1-2s)}.
\]
Since \(A_n\downarrow0\),
\[
\frac{\widetilde L_{n+1}}{\widetilde L_n}
=2^{-1/\theta}\left(\frac{A_{n+1}}{A_n}\right)^{1/(1-2s)}
<\frac12,
\]
so the binary construction of Section~\ref{sec:endpoint} applies verbatim.
Let \(\widetilde E\) and \(\widetilde\mu\) be the resulting set and Bernoulli
measure. For \(0<\nu\le\widetilde\mu\), set
\[
\widetilde K_\nu(x)
=
\int_{\widetilde E}
\frac{d\nu(y)}{|\ee{x}-\ee{y}|^2},
\qquad x\in\T\setminus\widetilde E.
\]  Repeating the proof of Proposition~\ref{prop:kernel}, with
\(L_n\) replaced by \(\widetilde L_n\), gives, for every nonzero
\(0<\nu\le\widetilde\mu\),
\[
\int_{\T\setminus\widetilde E}\widetilde K_\nu(x)^s\,dx
\gtrsim \nu(\T)
\sum_{n\ge0}2^{n(1-s)}\widetilde L_n^{1-2s}
\asymp \nu(\T)\sum_{n\ge0}A_n=\infty.
\]
Thus \(S_\nu'\notin H^s\) by Lemma~\ref{lem:boundary}.

Since
\(\widetilde L_n-2\widetilde L_{n+1}\asymp\widetilde L_n\),
the generation-\((n+1)\) deleted gaps have length comparable to
\(\widetilde L_n\). Therefore
\[
\sum_J |J|^\theta\left(\log\frac e{|J|}\right)^{q_c}
\lesssim 1+
\sum_{n\ge0}2^n\widetilde L_n^\theta
\left(\log\frac e{\widetilde L_n}\right)^{q_c}.
\]
Now
\[
2^n\widetilde L_n^\theta
=L_0^\theta\left(\frac{A_n}{A_0}\right)^{1/(1-s)},
\qquad
\log\frac e{\widetilde L_n}\lesssim n+1.
\]
Since \(q_c=s/(1-s)\), the last summand is bounded by a constant multiple of
\[
\frac1{(n+2)(\log(n+2))^{1/(1-s)}},
\]
whose series converges.  This proves the corollary.
\end{proof}

\begin{remark}
Corollary~\ref{cor:logintro} is intended only as a strengthening of the endpoint
example within the same binary Cantor scheme.  No optimality claim for the
logarithmic exponent is made here.
\end{remark}

\section{Preliminaries and a necessary energy condition}
\label{sec:preliminaries}

In Sections~\ref{sec:preliminaries}--\ref{sec:proof-counterexample} we
prove Theorem~\ref{thm:main-counterexample}.  The notation is reset here: the
analytic parameters \(s,q_c,\theta\) are not used below, and all sets, measures,
and scale sequences introduced from this point onward are independent of the
objects in Sections~\ref{sec:boundary}--\ref{sec:log}.  Fix \(m>3\), and allow constants
to depend on \(m\) and on fixed geometric parameters, but never on the scale
index.  The symbols \(c,C>0\) may change from line to line; all comparability
constants are uniform in the generation, the cylinder arc, and the selected
Whitney box.

Let \(U=u_{\max}\) be the maximal solution of
\eqref{eq:intro-semilinear}.
Its existence and maximality follow from the Keller--Osserman theory; see
\cite[Lemma~8.3]{IN24} and \cite{Kel57,Oss57}.  Since \(U>0\), it solves
\(\Delta U=U^m\).  The power nonlinearity satisfies the hypotheses of
Bandle--Marcus, whose boundary blow-up results give the asymptotic and
uniqueness used below; see \cite[Theorems~2.3--2.4]{BM92}.  In the notation of
the present equation, the same asymptotic is recorded in
\cite[Section~9, p.~2610]{IN24}.

Following \cite[Section~8B]{IN24}, a solution \(u\) is nearly maximal if
\[
\limsup_{r\to1^-}\int_{\T}(U-u)(r\xi)\,|d\xi|<\infty.
\]
By maximality, \(U-u\ge0\); moreover, it is subharmonic, and the measures
\((U-u)(r\xi)|d\xi|\) have a weak-* limit \(\mu[u]\), called the deficiency
measure.  The class of all such measures is denoted by \(\Def_m\).
The published theorem \cite[Theorem~1.3(i)]{IN24} states, for
$
\alpha=\frac{m-3}{m-1},
$
that
\[
\Def_m\subset\calM_\alpha(\T),
\qquad
\calM_\beta(\T)\subset\Def_m\quad(\beta<\alpha).
\]

Set
\begin{equation}\label{eq:alpha-gamma}
\gamma=\frac2{m-1},
\qquad
\alpha=1-\gamma=\frac{m-3}{m-1}.
\end{equation}
Then
\begin{equation}\label{eq:gamma-m}
\gamma m=\gamma+2.
\end{equation}
Write \(\delta(z)=1-|z|\).  The cited boundary asymptotic gives
\begin{equation}\label{eq:U-asymptotic}
U(z)\asymp\delta(z)^{-\gamma}
\qquad(z\to\T).
\end{equation}

For \(z\in\D\) and \(\xi\in\T\), set
\[
P(z,\xi)=\frac{1-|z|^2}{|z-\xi|^2}.
\]
For a finite positive measure \(\mu\) on \(\T\), define
\[
P\mu(z)=\frac1{2\pi}\int_{\T}P(z,\xi)\,d\mu(\xi).
\]
This normalization matches the use of \(|d\xi|\) in the weak boundary trace.
Let
\[
G(z,w)=\log\left|\frac{1-\bar w z}{z-w}\right|
\]
be the Green kernel of the disk.  For a positive Radon measure \(\lambda\) in
\(\D\), put
\[
\calG\lambda(z)=\frac1{2\pi}\int_{\D}G(z,w)\,d\lambda(w).
\]
Thus \(\Delta_zG(z,w)=-2\pi\delta_w\) and
\(\Delta\calG\lambda=-\lambda\).

Every solution satisfies \(u\le U\).  Put \(v=U-u\ge0\).  Then
\begin{equation}\label{eq:defect-equation}
\Delta v
=U^m-(U-v)_+^m
=:\calA(z,v),
\end{equation}
where the nonlinear defect \(\calA\) is defined by
\begin{equation}\label{eq:defect-definition}
\calA(z,t)=U(z)^m-(U(z)-t)_+^m,
\qquad t\ge0.
\end{equation}
Equivalently,
\begin{equation}\label{eq:linear-sign}
-\Delta v+\calA(z,v)=0.
\end{equation}

\begin{lemma}\label{lem:defect-estimates}
There are constants $c,C>0$ such that, whenever $z$ is sufficiently close to
$\T$ and $t\ge0$,
\begin{equation}
c\,\delta(z)^{-\gamma-2}\min\{t\delta(z)^\gamma,1\}
\le \calA(z,t)
\le C\,\delta(z)^{-\gamma-2}\min\{t\delta(z)^\gamma,1\}.
\label{eq:defect-two-sided}
\end{equation}
\end{lemma}

\begin{proof}
Using \eqref{eq:defect-definition} and
\(1-(1-\vartheta)_+^m\asymp\min\{\vartheta,1\}\) for \(\vartheta\ge0\),
\[
\calA(z,t)
=U(z)^m
\left[
1-\left(1-\frac{t}{U(z)}\right)_+^m
\right]
\asymp
U(z)^m\min\left\{\frac{t}{U(z)},1\right\}.
\]
The result follows from \eqref{eq:U-asymptotic} and
\eqref{eq:gamma-m}.
\end{proof}

To identify the harmonic part in the Riesz decomposition, we use the following
boundary-trace fact.

\begin{lemma}
\label{lem:green-zero-weak-trace}
Let \(\lambda\) be a positive Radon measure in \(\D\) with
\[
\int_{\D}\delta(w)\,d\lambda(w)<\infty.
\]
Then \(\calG\lambda\) has zero weak boundary trace: for every
\(\varphi\in C(\T)\),
\[
\lim_{r\to1^-}
\int_{\T}\calG\lambda(r\xi)\varphi(\xi)\,|d\xi|=0.
\]
\end{lemma}

\begin{proof}
For the disk Green function,
\[
\frac1{2\pi}\int_0^{2\pi}G(re^{i\phi},w)\,d\phi
=\log\frac1{\max\{r,|w|\}}.
\]
Tonelli's theorem gives
\[
\int_{\T}\calG\lambda(r\xi)\,|d\xi|
=\int_{\D}\log\frac1{\max\{r,|w|\}}\,d\lambda(w).
\]
It is enough to consider \(r\ge1/2\).  Then
\[
0\le \log\frac1{\max\{r,|w|\}}
\le
\log 2\,\mathbf 1_{\{|w|<1/2\}}
+C\delta(w)\,\mathbf 1_{\{|w|\ge1/2\}}.
\]
The right-hand side is \(\lambda\)-integrable, since \(\lambda\) is Radon and
\(\int_{\D}\delta\,d\lambda<\infty\).  Since the integrand tends
pointwise to zero as \(r\to1^-\), dominated convergence gives
\[
\int_{\T}\calG\lambda(r\xi)\,|d\xi|\to0.
\]
Therefore, for every \(\varphi\in C(\T)\),
\[
\left|
\int_{\T}\calG\lambda(r\xi)\varphi(\xi)\,|d\xi|
\right|
\le
\|\varphi\|_\infty
\int_{\T}\calG\lambda(r\xi)\,|d\xi|
\to0.
\]
\end{proof}

The following consequence of the Riesz decomposition is the analogue of
\cite[Lemma~8.4]{IN24} needed here.

\begin{lemma}[Poisson--Green representation and energy bound]
\label{lem:energy-necessary}
Let $u$ be a nearly maximal solution, and suppose its deficiency measure is
$\mu$.  Set $v=U-u$.  Then
\begin{equation}
v=P\mu-\calG\lambda,
\qquad
\lambda=\calA(z,v(z))\,dA(z),
\label{eq:Riesz}
\end{equation}
and
\begin{equation}
\int_\D \delta(z)\calA(z,v(z))\,dA(z)<\infty.
\label{eq:energy-finite}
\end{equation}
\end{lemma}

\begin{proof}
Define the radial mass function
\[
\Phi_v(r)=\int_{\T}v(r\xi)\,|d\xi|.
\]
By \eqref{eq:defect-equation}, \(v\) is subharmonic, so
\(\Phi_v(r)\) is nondecreasing in \(r\).
Near maximality gives
\[
\sup_{0<r<1}\Phi_v(r)<\infty.
\]
Hence \(v\) has a
finite least harmonic majorant \(h\) \cite[Theorem~3.6.6]{AG01}.  Since
\(\calA(\cdot,v)\) is continuous and nonnegative, the measure
\(\lambda=\calA(z,v(z))\,dA(z)\) is a positive Radon measure.  The Riesz
decomposition on the disk gives
\[
v=h-\calG\lambda,
\qquad
\lambda=\calA(z,v(z))\,dA(z);
\]
see \cite[Corollary~4.4.2(ii)]{AG01}.  Evaluating at zero yields
\(\calG\lambda(0)=h(0)-v(0)<\infty\).
Since
\(1-|w|\le-\log|w|=G(0,w)\) for \(0<|w|<1\), and
\[
\calG\lambda(0)
=
\frac1{2\pi}\int_{\D}G(0,w)\,d\lambda(w),
\]
we obtain
\[
\int_{\D}\delta(w)\,d\lambda(w)
\le
2\pi\calG\lambda(0)<\infty.
\]
As \(d\lambda(w)=\calA(w,v(w))\,dA(w)\), this is precisely
\eqref{eq:energy-finite}.

By Lemma~\ref{lem:green-zero-weak-trace}, \(\calG\lambda\) has zero weak
boundary trace.  Thus, for every \(\varphi\in C(\T)\),
\[
\lim_{r\to1^-}\int_\T h(r\xi)\varphi(\xi)\,|d\xi|
=\int_\T\varphi\,d\mu.
\]
Every positive harmonic function on \(\D\) has a unique finite measure
\(\rho\) such that
\[
h(z)=\int_{\T}\frac{1-|z|^2}{|z-\xi|^2}\,d\rho(\xi).
\]
Its weak boundary trace with respect to \(|d\xi|\) is \(2\pi\rho\).
Comparing with the preceding identity gives
\(\rho=\mu/(2\pi)\), and therefore \(h=P\mu\) with the normalization in
\eqref{eq:Riesz}.  This also fixes the normalization of \(\mu\) uniquely.
\end{proof}

\section{Minimal-fine localization and Whitney balls}
\label{sec:minimal-fine-localization}

We next use minimal thinness to locate points where $P\mu$ dominates the Green
potential in \eqref{eq:Riesz}.

All reductions below are lower-semicontinuously regularized reductions.  Thus,
if $h_0$ is a positive superharmonic function in $\D$ and $F\subset\D$ is Borel,
$\red_{h_0}^F$ denotes the lower-semicontinuous regularization of the infimum of
all positive superharmonic functions that dominate $h_0$ on $F$.  For
$\zeta\in\T$ write
\[
M_\zeta(z)=P(z,\zeta)=\frac{1-|z|^2}{|z-\zeta|^2},
\qquad M_\zeta(0)=1.
\]

\begin{proposition}[Na\"\i m reduction criterion]
\label{prop:naim-wiener}
Let \(F\subset\D\) be Borel and let \(\zeta\in\T\).  If \(F\) is minimally
thin at \(\zeta\), then
\[
\inf_V \red_{M_\zeta}^{F\cap V}(0)=0,
\]
where the infimum is taken over all Martin-topology neighbourhoods \(V\)
of \(\zeta\).
\end{proposition}

\begin{proof}
By the open-enlargement property \cite[Lemma~9.2.2(c)]{AG01}, \(F\) is
contained in an open set \(\mathcal O\) that is minimally thin at \(\zeta\).
The reduction criterion \cite[Theorem~9.2.5]{AG01} implies that
\(\red_{M_\zeta}^{\mathcal O\cap V}(0)\)
is arbitrarily small as \(V\) ranges over the Martin neighbourhoods of
\(\zeta\).  Since \(F\cap V\subset \mathcal O\cap V\), the result follows by
monotonicity of reduction.
\end{proof}

For a domain \(\Omega\), a point \(z\in\Omega\), and a Borel set
\(A\subset\partial\Omega\), we denote by \(\omega_\Omega^z(A)\) the
harmonic measure of \(A\) in \(\Omega\) with pole at \(z\).

\begin{lemma}[Uniform reduction of Whitney balls]
\label{lem:uniform-reduction}
Fix constants $c_*>2\tau>0$.  There are constants $r_0>0$ and $c_W>0$,
depending only on $c_*$ and $\tau$, such that for every $\zeta\in\T$ and every
$0<r<r_0$, if
\[
W_{\zeta,r}=\overline{B\bigl((1-c_*r)\zeta,\tfrac12\tau r\bigr)},
\]
then
\[
\red_{M_\zeta}^{W_{\zeta,r}}(0)\ge c_W.
\]
\end{lemma}

\begin{proof}
By rotation, assume \(\zeta=1\).  On
\[
W_{1,r}=\overline{B(1-c_*r,\tfrac12\tau r)}
\]
one has
$
1-|z|\asymp r,
|1-z|\asymp r,
$
and hence
\[
M_1(z)=\frac{1-|z|^2}{|1-z|^2}\asymp r^{-1}
\qquad (z\in W_{1,r}).
\]
Therefore, by monotonicity and homogeneity of reduction,
\[
\red_{M_1}^{W_{1,r}}(0)
\ge
\red_{c r^{-1}}^{W_{1,r}}(0)
=
c r^{-1}\red_1^{W_{1,r}}(0).
\]
The reduced function \(\red_1^{W_{1,r}}\) is the equilibrium potential of the
compact obstacle \(W_{1,r}\), so
\[
\red_1^{W_{1,r}}(0)
=
\omega_{\D\setminus W_{1,r}}^0(\partial W_{1,r}).
\]

It remains to show that this harmonic measure is \(\gtrsim r\), uniformly
for small \(r\).  Put
\[
\Psi_r(z)=r^{-1}i\frac{1-z}{1+z}.
\]
This map sends \(\D\) conformally onto \(\Hh\), sends \(0\) to
\(i/r\), and sends \(W_{1,r}\) to a compact set \(K_r\Subset\Hh\).
Indeed, writing \(z=1-r\omega\), we have
\[
\Psi_r(1-r\omega)=\frac{i\omega}{2-r\omega},
\qquad |\omega-c_*|\le \frac{\tau}{2}.
\]
Let
\[
D_*:=\overline{B(c_*,\tau/2)},
\qquad f_r(\omega):=\frac{i\omega}{2-r\omega}.
\]
After first decreasing \(r_0\) so that
\(r(c_*+\tau/2)<1\), each \(f_r\) is a holomorphic injective fractional
linear map on a neighbourhood of \(D_*\).  The maps \(f_r\) converge
uniformly on \(D_*\) to \(f_0(\omega)=i\omega/2\), and
\[
f_0(D_*)=\overline{B(ic_*/2,\tau/4)}.
\]
Define the fixed disk
\[
K_-:=\overline{B(ic_*/2,\tau/8)}\Subset\Hh.
\]
The last compact containment follows from \(c_*>2\tau\).  Moreover,
\(K_-\) has distance \(\tau/8\) from \(f_0(\partial D_*)\).  Uniform
convergence on \(\partial D_*\), followed by invariance of the winding number
(or the argument principle), therefore gives
\[
K_-\subset f_r(D_*)=K_r=\Psi_r(W_{1,r})
\qquad(0<r<r_0)
\]
after decreasing \(r_0\).  Thus every \(K_r\) contains the same fixed
sub-obstacle \(K_-\).

Let \(R=1/r\), and let \((B_t)_{t\ge0}\) be planar Brownian motion in
\(\Hh\), started at \(iR\).  For a closed set \(A\), write
\(T_A=\inf\{t\ge0:B_t\in A\}\).  Take
\[
Y=\frac{c_*}{2}+\frac{\tau}{2},
\qquad
I=\{x+iY:|x|\le\tau/8\}.
\]
Thus the horizontal line \(L_Y=\{\Im z=Y\}\) lies strictly above \(K_-\).
The first-hit distribution on \(L_Y\), for Brownian motion started at
\(iR\), has density
\[
\frac{R-Y}{\pi(x^2+(R-Y)^2)}\,dx.
\]
For \(R\ge2Y+1\), integrating this density over \(|x|\le\tau/8\) gives
\[
\mathbb P^{iR}\{B_{T_{L_Y}}\in I\}
\ge \frac{c_{\mathrm{hit}}}{R},
\]
where \(c_{\mathrm{hit}}>0\) depends only on \(c_*\) and \(\tau\).  Moreover,
\[
c_{\mathrm{obs}}:=
\inf_{w\in I}
\omega_{\Hh\setminus K_-}^{w}(\partial K_-)>0,
\]
by positivity and continuity of harmonic measure on the compact interval
\(I\); here we use that the disk \(K_-\) is nonpolar.  The strong Markov
property at the first hitting time of \(L_Y\)
therefore yields
\[
\mathbb P^{iR}\{T_{K_-}<T_{\R}\}
\ge c_{\mathrm{hit}}c_{\mathrm{obs}}R^{-1}.
\]
Since
\(K_-\subset K_r\), conformal invariance and monotonicity of
obstacle-hitting probabilities now give
\[
\omega_{\D\setminus W_{1,r}}^0(\partial W_{1,r})
=
\omega_{\Hh\setminus K_r}^{iR}(\partial K_r)
\ge cR^{-1}
=cr.
\]
Combining the two estimates gives
\[
\red_{M_1}^{W_{1,r}}(0)\ge c_W>0.
\]
Rotation invariance completes the proof.
\end{proof}

Consequently, a minimally thin set cannot contain a full sequence of shrinking
Whitney balls.

\begin{lemma}
\label{lem:minthin-no-whitney}
Let \(F\subset\D\) be minimally thin at \(\zeta\in\T\), let
\(r_n\to0\), and fix \(c_*>2\tau>0\).  Then
\[
\overline{B\bigl((1-c_*r_n)\zeta,\tfrac12\tau r_n\bigr)}\subset F
\]
can hold for only finitely many indices \(n\).
\end{lemma}

\begin{proof}
Suppose that the inclusion holds along a subsequence, and denote the
corresponding balls by \(W_j\).  Then
\[
\sup_{z\in W_j}|z-\zeta|
\le
\left(c_*+\frac{\tau}{2}\right)r_{n_j}
\to0.
\]
Since the Martin compactification of \(\D\) is the closed disk,
every Martin neighbourhood \(V\) of \(\zeta\) contains \(W_j\) for all
sufficiently large \(j\).  Since \(W_j\subset F\), we have
\(W_j\subset F\cap V\).  Hence, by monotonicity of reduction and
Lemma~\ref{lem:uniform-reduction},
\[
\red_{M_\zeta}^{F\cap V}(0)
\ge \red_{M_\zeta}^{W_j}(0)
\ge c_W>0.
\]
Since \(V\) was arbitrary, taking the infimum over all Martin
neighbourhoods of \(\zeta\) contradicts
Proposition~\ref{prop:naim-wiener}.
\end{proof}

We can now combine the relative Fatou theorem with
Lemma~\ref{lem:minthin-no-whitney} to obtain the localization statement needed
later on.

\begin{proposition}[Fatou--Na\"\i m--Doob theorem]
\label{prop:FND}
Let \(\mu\) be a nonzero finite positive measure on \(\T\), and let
\(h=P\mu\). Let \(\lambda\) be a positive Radon measure in \(\D\)
satisfying
\[
\int_\D \delta(z)\,d\lambda(z)<\infty.
\]
Assume that \(0\le\calG\lambda\le h\) in \(\D\). Then
\[
\frac{\calG\lambda(z)}{h(z)}\longrightarrow0
\]
as \(z\to\zeta\) along the minimal fine filter, for
\(\mu\)-almost every \(\zeta\in\T\).
\end{proposition}

\begin{proof}
This is the Fatou--Na\"\i m--Doob theorem in potential form: if
\(h>0\) is harmonic and \(\calG\lambda\) is a Green potential, then
\((\calG\lambda)/h\) has minimal-fine limit zero at almost every
minimal Martin boundary point with respect to the representing measure
of \(h\); see \cite[Theorem~4.2]{Doo59} and
\cite[Section~9.4]{AG01}. In the disk, the minimal Martin boundary is
\(\T\), and the representing measure of our normalized \(h=P\mu\) is
\(\mu/(2\pi)\), which has the same null sets as \(\mu\).
\end{proof}

\begin{lemma}[Minimal-fine Whitney localization]
\label{lem:mf-whitney}
Let \(\mu\) be a nonzero finite positive measure on \(\T\), and let
\(\lambda\) be a positive Radon measure in \(\D\) such that
\[
\int_{\D}\delta(z)\,d\lambda(z)<\infty,
\qquad
0\le\calG\lambda\le P\mu.
\]
Let \(r_n\to0\) and fix \(c_*>2\tau>0\).  After discarding finitely many
terms, assume that the balls below are compactly contained in \(\D\),
and put
\[
\widetilde B_{\zeta,n}
=\overline{B\bigl((1-c_*r_n)\zeta,\tfrac12\tau r_n\bigr)}.
\]
For every \(\varepsilon>0\), there are measurable sets
\(\Gamma_{\varepsilon,N}\subset\T\) such that
\(\bigcup_{N\ge1}\Gamma_{\varepsilon,N}\) has full \(\mu\)-measure and,
whenever \(\zeta\in \Gamma_{\varepsilon,N}\) and \(n\ge N\), the ball
\(\widetilde B_{\zeta,n}\) contains a point \(z\) satisfying
\begin{equation}\label{eq:mf-good-point}
\calG\lambda(z)\le\varepsilon P\mu(z).
\end{equation}
\end{lemma}

\begin{proof}
For \(\varepsilon>0\), set
\(F_\varepsilon
=\{z\in\D:\calG\lambda(z)>\varepsilon P\mu(z)\}\).
By Proposition~\ref{prop:FND} and the definition of minimal-fine
convergence, \(F_\varepsilon\) is minimally thin at
\(\mu\)-almost every \(\zeta\in\T\).

For fixed \(w\), the Green kernel \(G(\cdot,w)\) is nonnegative and
lower semicontinuous, with value \(+\infty\) at \(w\). Fatou's lemma
therefore shows that \(\calG\lambda\) is lower semicontinuous. Since
\(P\mu\) is continuous,
\(\calG\lambda-\varepsilon P\mu\) is lower semicontinuous, and hence
\(F_\varepsilon\) is open.

At every such boundary point,
Lemma~\ref{lem:minthin-no-whitney} gives
$
\widetilde B_{\zeta,n}\not\subset F_\varepsilon
$
for all sufficiently large \(n\).
Define
\[
\Gamma_{\varepsilon,N}=\bigcap_{n\ge N}
\{\zeta:\widetilde B_{\zeta,n}\not\subset F_\varepsilon\}.
\]
The set where \(\widetilde B_{\zeta,n}\subset F_\varepsilon\) is open in \(\zeta\):
indeed, a compact subset of the open set \(F_\varepsilon\) has positive distance
from \(\D\setminus F_\varepsilon\), while the ball depends continuously on
\(\zeta\) in the Hausdorff metric.  Hence every \(\Gamma_{\varepsilon,N}\) is Borel.  The
sets \(\Gamma_{\varepsilon,N}\) increase with \(N\), and the preceding eventual
noncontainment shows that their union has full \(\mu\)-measure.  Finally,
\(\widetilde B_{\zeta,n}\not\subset F_\varepsilon\) furnishes a point outside
\(F_\varepsilon\), which is exactly \eqref{eq:mf-good-point}.
\end{proof}

We shall use the following scale-invariant Harnack-chain estimate.

\begin{lemma}[Uniform Harnack-chain estimate]
\label{lem:harnack-chain}
Fix constants \(C_{\mathfrak q}<\infty\), \(C_{\mathrm{ch}}\ge1\), and
\(0<r_+<\infty\).  There is a constant \(H>1\) with the following property.
Let \(j\ge1\) be an integer and \(0<\ell<1\), and suppose that
\(w\in C^2\) is positive and solves
\[
-\Delta w+\mathfrak q(z)w=0
\]
in a union of balls \(B_k=B(x_k,r_k)\), \(1\le k\le N\), where
\(N\le C_{\mathrm{ch}}j\),
\[
0<r_k\le r_+\ell,
\]
the endpoints satisfy \(x\in\frac12B_1\), \(y\in\frac12B_N\), and
\[
\frac12B_k\cap\frac12B_{k+1}\ne\varnothing
\qquad(1\le k<N).
\]
Assume that \(\mathfrak q\) is measurable and that
\[
0\le \mathfrak q(z)\le C_{\mathfrak q}\ell^{-2}
\]
in the chain.  Then
\[
w(y)\ge H^{-j}w(x).
\]
The constant \(H\) depends only on \(C_{\mathfrak q},C_{\mathrm{ch}},r_+\), not on
\(\ell,j,x,y\).
\end{lemma}

\begin{proof}
Writing the equation as \(\Delta w-\mathfrak q w=0\), the zero-order coefficient is
nonpositive.  After the change of variables \(z=x_k+r_k\zeta\), the equation
on the unit ball has coefficient \(r_k^2\mathfrak q(x_k+r_k\zeta)\), bounded above by
\(r_+^2C_{\mathfrak q}\).  Thus the interior Harnack inequality
\cite[Theorem~8.20]{GT01} gives
\[
\sup_{\frac12B_k}w\le C_H\inf_{\frac12B_k}w
\]
with \(C_H\) depending only on \(C_{\mathfrak q},r_+\).  Since consecutive half-balls
intersect, iteration over \(N\le C_{\mathrm{ch}}j\) balls gives
\(w(y)\ge C_H^{-C_{\mathrm{ch}}j}w(x)\).  Taking \(H=C_H^{C_{\mathrm{ch}}}\) proves the claim.
\end{proof}

\section{A Moran construction at the endpoint}
\label{sec:moran-construction}

We now build the measure that will contradict the endpoint correspondence.  The
parameters are chosen so that the geometric series controlling the
$\alpha$-Beurling--Carleson condition converges, while the series driving the
PDE energy diverges.

Choose a logarithmic exponent
\begin{equation}
1+\alpha<\mathsf B\le2.
\label{eq:B-choice}
\end{equation}
For $n\ge3$, let
\begin{equation}
a_n=\frac1{n(\log n)^{\mathsf B}}.
\label{eq:a-n}
\end{equation}
Then
\begin{equation}
\sum_{n\ge3}a_n<\infty.
\label{eq:a-sum}
\end{equation}
Fix the Whitney-localization constants
\begin{equation}
c_*=8,
\qquad \tau=1.
\label{eq:cstar-tau-fixed}
\end{equation}
In the final energy estimate we shall use boxes lying in the strip
$6\ell\le\delta\le10\ell$ and Harnack chains of length $O(j)$.  Let $H>1$ be a
Harnack constant supplied by Lemma~\ref{lem:harnack-chain} for this fixed
geometry and for a zero-order coefficient bounded by \(C\ell^{-2}\), as follows from
\eqref{eq:U-asymptotic}.  Fix \(\kappa>0\) so small that
\begin{equation}
\kappa\log H<1.
\label{eq:kappa-small}
\end{equation}
For $n\ge3$ set
\begin{equation}
k_n=\max\left\{1,\left\lfloor \kappa\log\frac1{a_n}\right\rfloor\right\}.
\label{eq:k-n}
\end{equation}
The integer \(k_n\) is the number of Harnack boxes assigned to each
generation-\(n\) private gap.  By \eqref{eq:a-n} and \eqref{eq:k-n},
\(k_n\asymp\log n\).  Hence
\begin{equation}
\sum_{n\ge3}a_nk_n^\alpha<\infty,
\label{eq:a-kalpha-sum}
\end{equation}
since \eqref{eq:B-choice} gives $\mathsf B-\alpha>1$, whereas
\begin{equation}
\sum_{n\ge3}a_nk_n=\infty,
\label{eq:a-k-diverge}
\end{equation}
since \eqref{eq:B-choice} also gives $\mathsf B-1\le1$.

We now construct a compact set $E\subset\T$.  Work first inside a fixed closed
arc $I_*\subset\T$ of length $|I_*|<10^{-3}$.  Put \(\calI_2=\{I_*\}\), \(N_2=1\), and \(\ell_2=|I_*|\).
Assume that the generation $n-1$ family $\calI_{n-1}$ has already been
constructed; it consists of $N_{n-1}$ pairwise disjoint closed arcs, all of
length $\ell_{n-1}$.  Choose a large integer $b_n\ge2$ and define
\begin{equation}
N_n=N_{n-1}b_n,
\qquad
\ell_n=\left(\frac{a_n}{N_n}\right)^{1/\alpha}.
\label{eq:N-ell}
\end{equation}
Let
\begin{equation}
g_n=100k_n\ell_n.
\label{eq:g-n}
\end{equation}
Write a parent as \(I=[x_I,x_I+\ell_{n-1}]\).  For \(1\le i\le b_n\), define
\[
I_{n,I,i}
=
[x_I+(i-1)(\ell_n+g_n),\,x_I+(i-1)(\ell_n+g_n)+\ell_n],
\]
and the following private gap
\[
\Gap_{n,I,i}
=
(x_I+(i-1)(\ell_n+g_n)+\ell_n,\,x_I+i(\ell_n+g_n)).
\]
The terminal residual gap is
\[
R_{n,I}=(x_I+b_n(\ell_n+g_n),\,x_I+\ell_{n-1}).
\]
Set
\[
\calI_n=\{I_{n,I,i}:I\in\calI_{n-1},\ 1\le i\le b_n\}.
\]
Complementary arcs of \(E\) are obtained, up to endpoints, by coalescing
these private and residual gaps.

The integer $b_n$ can be chosen so that
\begin{equation}
b_n(\ell_n+g_n)\le \frac12\ell_{n-1},
\qquad
N_n\ell_n\le 2^{-n},
\qquad
\ell_n\le 2^{-20}\ell_{n-1}.
\label{eq:packing-requirements}
\end{equation}
Indeed,
\[
b_n(\ell_n+g_n)
\le C k_n b_n\ell_n
= C k_n a_n^{1/\alpha}N_{n-1}^{-1/\alpha}b_n^{1-1/\alpha}
\longrightarrow0
\]
as $b_n\to\infty$, because $1-1/\alpha<0$.  Also
$N_n\ell_n=a_n^{1/\alpha}N_n^{1-1/\alpha}\to0$ as $N_n\to\infty$, again because
$1-1/\alpha<0$.  Finally, $\ell_n=(a_n/N_n)^{1/\alpha}\to0$ as
$N_n\to\infty$.

Let $E_n$ be the union of the arcs in $\calI_n$, and define
\[
E=\bigcap_{n\ge3}E_n.
\]
Then $E$ is compact and, by \eqref{eq:packing-requirements},
$|E_n|=N_n\ell_n\le2^{-n}$.  Hence $|E|=0$.

\begin{lemma}
\label{lem:E-alpha-BC}
The set $E$ is an $\alpha$-Beurling--Carleson set.
\end{lemma}

\begin{proof}
Let \(\mathcal R\) be the family consisting of the initial outer gap, all
private gaps, and all residual gaps.  Its members are disjoint up to endpoints,
and no later generation enters a gap that has already been removed.
Consequently, up to endpoints, every complementary arc of \(E\) is obtained by
coalescing adjacent members of \(\mathcal R\).  Since
\(0<\alpha<1\),
$
(a+b)^\alpha\le a^\alpha+b^\alpha,
$
and therefore
\[
\sum_{J\subset\T\setminus E}|J|^\alpha
\le
\sum_{R\in\mathcal R}|R|^\alpha.
\]
At generation \(n\), there are \(N_n\) private gaps, each of length
\(g_n\) given by \eqref{eq:g-n}.  By \eqref{eq:a-kalpha-sum}, their total
contribution is
\[
100^\alpha\sum_{n\ge3}N_nk_n^\alpha\ell_n^\alpha
=
100^\alpha\sum_{n\ge3}a_nk_n^\alpha
<\infty.
\]
There is one residual gap per generation-\((n-1)\) parent, each of length at
most \(\ell_{n-1}\).  Hence, by \eqref{eq:a-sum}, their contribution is at most
\[
|I_*|^\alpha+\sum_{n\ge4}N_{n-1}\ell_{n-1}^\alpha
=
|I_*|^\alpha+\sum_{n\ge4}a_{n-1}
<\infty.
\]
The outer gap contributes one finite term.
\end{proof}

Define the Moran probability measure $\mu$ as follows.  Each generation $n$ arc
has mass
\begin{equation}
\mu(I)=\frac1{N_n},
\qquad I\in\calI_n.
\label{eq:mu-cylinder}
\end{equation}
The cylinder sets $E\cap I$, $I\in\calI_n$, form a partition of $E$,
and these masses are consistent under refinement.  They therefore define a
probability measure supported on $E$.
Since $b_n\ge2$, one has $N_n\to\infty$, and the cylinder masses tend uniformly
to zero; hence $\mu$ is nonatomic.

By Lemma~\ref{lem:E-alpha-BC}, $\mu$ is supported on a single
$\alpha$-Beurling--Carleson set.  Thus
\begin{equation}
\mu\in\calM_\alpha(\T).
\label{eq:mu-in-Malpha}
\end{equation}
For a generation $n$ arc $I$, we have
\begin{equation}
\frac{\mu(I)}{|I|^\alpha}
=\frac{1/N_n}{\ell_n^\alpha}
=\frac1{a_n}=:D_n.
\label{eq:D-n}
\end{equation}
Thus $D_n\to\infty$, and by \eqref{eq:k-n},
\begin{equation}
k_n\le \kappa\log D_n
\label{eq:k-log-D}
\end{equation}
for all sufficiently large $n$.

\section{Proof of the counterexample}
\label{sec:proof-counterexample}

We now combine the endpoint measure constructed in
Section~\ref{sec:moran-construction} with the localization machinery from
Section~\ref{sec:minimal-fine-localization}.

\begin{proof}[Proof of Theorem~\ref{thm:main-counterexample}]
Assume, for contradiction, that the measure $\mu$ constructed above is the
deficiency measure of a nearly maximal solution $u$.  Set \(v=U-u\).
By Lemma~\ref{lem:energy-necessary}, the representation
\eqref{eq:Riesz} holds, and
\begin{equation}
\int_\D\delta(z)\calA(z,v(z))\,dA(z)<\infty.
\label{eq:assume-finite-energy}
\end{equation}
Because \(v\ge0\) and \(v=P\mu-\calG\lambda\), we also have
\begin{equation}
0\le\calG\lambda\le P\mu
\qquad\text{in }\D,
\label{eq:green-below-poisson}
\end{equation}
Together, \eqref{eq:assume-finite-energy} and
\eqref{eq:green-below-poisson} verify the hypotheses of
Lemma~\ref{lem:mf-whitney}.  We shall
derive a contradiction to \eqref{eq:assume-finite-energy}.

Apply Lemma~\ref{lem:mf-whitney}, with \(\varepsilon=1/2\), to a
sufficiently far tail of \((\ell_n)_{n\ge3}\), retaining the original
indexing. Thus, for some \(n_{\mathrm{loc}}\ge3\), there are increasing
measurable sets \(\Gamma_{1/2,N}\), \(N\ge n_{\mathrm{loc}}\), whose
union has full \(\mu\)-measure and for which the conclusion of the lemma
holds whenever \(n\ge N\). Since \(\mu(E)=1\), continuity from below
yields \(n_*\ge n_{\mathrm{loc}}\) such that
\[
\mu\bigl(\Gamma_{1/2,n_*}\cap E\bigr)>0.
\]
Set \(\Gamma=\Gamma_{1/2,n_*}\cap E\) and \(n_0=n_*\).
Since $\mu$ is supported on $E$,
\begin{equation}
\mu(\Gamma)>0.
\label{eq:Gamma-positive}
\end{equation}
For every $\zeta\in\Gamma$ and every $n\ge n_0$, there exists
\begin{equation}
z_{\zeta,n}\in
\overline{B\bigl((1-c_*\ell_n)\zeta,\tfrac12\tau\ell_n\bigr)}
\label{eq:selected-whitney-point-general}
\end{equation}
such that
\begin{equation}
\calG\lambda(z_{\zeta,n})\le\frac12P\mu(z_{\zeta,n}).
\label{eq:green-half-on-gamma}
\end{equation}
In particular,
\begin{equation}
(c_*-\tfrac12\tau)\ell_n
\le \delta(z_{\zeta,n})
\le (c_*+\tfrac12\tau)\ell_n.
\label{eq:good-point-height}
\end{equation}

For large \(n\), let
\(\calI_n(\Gamma)=\{I\in\calI_n:\mu(I\cap\Gamma)>0\}\).
Since the generation \(n\) arcs form a partition modulo endpoint sets of
\(\mu\)-measure zero and, by \eqref{eq:mu-cylinder}, each has mass \(1/N_n\),
\[
\mu(\Gamma)=\sum_{I\in\calI_n(\Gamma)}\mu(\Gamma\cap I)
\le\frac{\#\calI_n(\Gamma)}{N_n}.
\]
Consequently,
\begin{equation}
\#\calI_n(\Gamma)\ge \mu(\Gamma)N_n.
\label{eq:good-arcs-count}
\end{equation}
For each $I\in\calI_n(\Gamma)$ choose $\zeta_I\in I\cap\Gamma$, and then choose
$z_I=z_{\zeta_I,n}$ as in \eqref{eq:selected-whitney-point-general}.  By
\eqref{eq:green-half-on-gamma},
\begin{equation}
v(z_I)=P\mu(z_I)-\calG\lambda(z_I)
\ge \frac12P\mu(z_I).
\label{eq:v-ge-half-P}
\end{equation}
Moreover, by \eqref{eq:good-point-height} and because $\zeta_I\in I$ while
$|I|=\ell_n$, for every $\xi\in I$,
\begin{equation}
\delta(z_I)\asymp\ell_n,
\qquad |z_I-\xi|\le C\ell_n.
\label{eq:zI-geometry}
\end{equation}
For all large \(n\), \(|z_I|\ge1/2\); moreover, \eqref{eq:zI-geometry}
gives
\[
1-|z_I|^2=(1+|z_I|)\delta(z_I)\ge c\ell_n,
\qquad |z_I-\xi|^2\le C\ell_n^2.
\]
Thus the Poisson kernel satisfies \(P(z_I,\xi)\ge c/\ell_n\) for every
\(\xi\in I\), and hence
\begin{equation}
P\mu(z_I)
\ge c\frac{\mu(I)}{\ell_n}
=c\frac1{N_n\ell_n}.
\label{eq:P-lower}
\end{equation}
Using \eqref{eq:N-ell}, \eqref{eq:alpha-gamma}, and \eqref{eq:D-n}, we get
\begin{equation}
\frac1{N_n\ell_n}
=\frac1{a_n}\ell_n^{-\gamma}
=D_n\ell_n^{-\gamma}.
\label{eq:density-amplification}
\end{equation}
Combining \eqref{eq:v-ge-half-P}, \eqref{eq:P-lower}, and
\eqref{eq:density-amplification}, we arrive at
\begin{equation}
v(z_I)\ge cD_n\ell_n^{-\gamma}.
\label{eq:v-large-at-I}
\end{equation}

We propagate this lower bound from $z_I$ into the private gap following $I$.  Let $\Gap_I$ be the private gap immediately
after the generation $n$ arc $I$.  It has length $100k_n\ell_n$.  Identifying
$\Gap_I$ with the interval $(0,100k_n\ell_n)$ so that $0$ is the endpoint adjacent
to $I$, let $J_{I,j}$ be the subarc corresponding to
\[
\big[(6j-2)\ell_n,6j\ell_n\big],
\qquad 1\le j\le k_n.
\]
Then the arcs $J_{I,j}$ have length $2\ell_n$, are separated from each other
and from the endpoints of $\Gap_I$ by at least $4\ell_n$, and satisfy
$\operatorname{dist}(I,J_{I,j})\le Cj\ell_n$.

Let \(J'_{I,j}\) be the middle half of \(J_{I,j}\), and let
\(\xi_{I,j}\) be the midpoint of \(J'_{I,j}\).  Define
\[
Q_{I,j}
=\{r\xi:\xi\in J_{I,j},\ 7\ell_n<1-r<9\ell_n\}
\]
and
\[
Q'_{I,j}
=\{r\xi:\xi\in J'_{I,j},\ \tfrac{15}{2}\ell_n<1-r<\tfrac{17}{2}\ell_n\}.
\]
Put \(w_{I,j}=(1-8\ell_n)\xi_{I,j}\).  For all sufficiently large \(n\),
\(Q'_{I,j}\Subset Q_{I,j}\) and, uniformly in \(I\) and \(j\),
\begin{equation}
\operatorname{dist}\bigl(Q'_{I,j},\D\setminus Q_{I,j}\bigr)
\ge c\ell_n.
\label{eq:box-buffer}
\end{equation}
Indeed, the radial and angular margins are both \(\ell_n/2\), and the
radial coordinate is bounded below by \(1/2\).  Also,
\begin{equation}
6\ell_n\le\delta(z)\le10\ell_n,
\qquad z\in Q_{I,j},
\label{eq:box-height}
\end{equation}
and, since \(|J'_{I,j}|=\ell_n\),
\begin{align}
|Q'_{I,j}|
&=\pi|J'_{I,j}|
\left[\left(1-\frac{15}{2}\ell_n\right)^2
-\left(1-\frac{17}{2}\ell_n\right)^2\right] \notag\\
&\asymp\ell_n^2.
\label{eq:box-area}
\end{align}

Define the linearization coefficient along \(v\) by
\[
\mathfrak q_v(z)=\int_0^1 m\,(U(z)-\vartheta v(z))_+^{m-1}\,d\vartheta.
\]
Then
\[
\calA(z,v(z))
=\int_0^{v(z)}m\,(U(z)-t)_+^{m-1}\,dt
=\mathfrak q_v(z)v(z),
\]
and
\[
0\le\mathfrak q_v(z)\le mU(z)^{m-1}.
\]
Thus \(\mathfrak q_v\in L^\infty_{\mathrm{loc}}(\D)\).  Together with
\eqref{eq:linear-sign}, the preceding identity gives
$
-\Delta v+\mathfrak q_v v=0.
$
The deficiency trace of \(v\) is the nonzero measure \(\mu\), so
\(v\not\equiv0\).  Since \(v\ge0\), the strong maximum principle gives
\(v>0\) in \(\D\).

For all sufficiently large \(n\), \eqref{eq:U-asymptotic} and
\eqref{eq:box-height} give
\[
0\le\mathfrak q_v(z)\le C\ell_n^{-2}
\]
on every chain contained in the strip \eqref{eq:box-height}.  By
\eqref{eq:cstar-tau-fixed}, the selected point \(z_I\) also satisfies
\[
\frac{15}{2}\ell_n\le\delta(z_I)\le\frac{17}{2}\ell_n.
\]
Since \(|z_I-(1-8\ell_n)\zeta_I|\le\ell_n/2\) and \(|z_I|\ge1/2\),
the angular distance between \(z_I/|z_I|\) and \(\zeta_I\) is \(O(\ell_n)\).
Together with \(\operatorname{dist}(I,J_{I,j})\le Cj\ell_n\), this shows
that the following path has length \(O(j\ell_n)\): first join \(z_I\)
radially to the circle \(\delta=8\ell_n\), and then follow the shorter
circular arc to \(w_{I,j}\).  Cover it by balls of
radius \(\ell_n/4\) whose successive centres are at distance at most
\(\ell_n/8\).  These balls stay in
\(6\ell_n<\delta<10\ell_n\), their half-balls overlap, and their number is
\(O(j)\).  Lemma~\ref{lem:harnack-chain}, applied with
\(\mathfrak q=\mathfrak q_v\) and with the fixed constant \(H\) chosen
before the construction, therefore gives
\begin{equation}
v(w_{I,j})\ge H^{-j}v(z_I).
\label{eq:v-Harnack-propagation}
\end{equation}
By \eqref{eq:kappa-small} and \eqref{eq:k-log-D},
\[
H^{-k_n}D_n\ge D_n^{1-\kappa\log H}
\]
for all sufficiently large $n$.  In particular, there is $c_{\mathrm{HD}}>0$ such that
\begin{equation}
H^{-k_n}D_n\ge c_{\mathrm{HD}}
\label{eq:HD-lower}
\end{equation}
for all large $n$.  Combining \eqref{eq:v-large-at-I},
\eqref{eq:v-Harnack-propagation}, and \eqref{eq:HD-lower}, we obtain
\begin{equation}
v(w_{I,j})\ge c\ell_n^{-\gamma},
\qquad 1\le j\le k_n.
\label{eq:v-lower-centers}
\end{equation}
Starting from \eqref{eq:v-lower-centers}, join any \(z\in Q'_{I,j}\) to
\(w_{I,j}\) by a radial segment
and then a circular arc inside \(Q'_{I,j}\).  This path has length at most
\(C\ell_n\).  By \eqref{eq:box-buffer}, it can be covered by \(O(1)\)
balls of radius \(c_{\mathrm{ball}}\ell_n\), where \(c_{\mathrm{ball}}>0\) is fixed, with overlapping
half-balls and with every ball contained in \(Q_{I,j}\).  Applying
Lemma~\ref{lem:harnack-chain} once
more and absorbing the fixed loss into the constant gives
\begin{equation}
v(z)\ge c\ell_n^{-\gamma},
\qquad z\in Q'_{I,j}.
\label{eq:v-lower-boxes}
\end{equation}
For \(z\in Q'_{I,j}\),
\eqref{eq:box-height} and \eqref{eq:v-lower-boxes} give
\(v(z)\delta(z)^\gamma\ge c>0\).
Consequently the minimum in \eqref{eq:defect-two-sided} is bounded below by a
positive constant independent of \(n,I,j\).  Lemma~\ref{lem:defect-estimates}
therefore yields
\begin{equation}
\calA(z,v(z))\ge c\delta(z)^{-\gamma-2}
\ge c\ell_n^{-\gamma-2},
\qquad z\in Q'_{I,j}.
\label{eq:defect-lower-boxes}
\end{equation}
By \eqref{eq:box-height}, \eqref{eq:box-area}, and
\eqref{eq:defect-lower-boxes}, each selected box contributes at least
\begin{align}
\int_{Q'_{I,j}}\delta(z)\calA(z,v(z))\,dA(z)
&\ge c\ell_n\cdot\ell_n^{-\gamma-2}\cdot\ell_n^2 \notag\\
&=c\ell_n^{1-\gamma}=c\ell_n^\alpha.
\label{eq:one-box-energy}
\end{align}

Increasing \(n_0\) if necessary, we may assume that
\eqref{eq:k-log-D}, \eqref{eq:HD-lower}, the Whitney localization, and all
geometric comparability estimates hold for every \(n\ge n_0\).

The boxes \(Q'_{I,j}\) are pairwise disjoint.  Indeed, their angular bases
are compactly contained in private gaps.  Private gaps of the same
generation are disjoint by construction; if a private gap is created at
generation \(n\), all later descendants are placed inside the remaining
generation \(n\) child arcs and hence are disjoint from that gap.  Within
one private gap, the intervals \(J_{I,j}\) are separated.  The shrinkage
to \(Q'_{I,j}\) removes possible boundary contacts.  Using
\eqref{eq:Gamma-positive} and summing \eqref{eq:one-box-energy} over all
sufficiently large $n$ and over
$I\in\calI_n(\Gamma)$ gives
\begin{align}
\int_\D\delta(z)\calA(z,v(z))\,dA(z)
&\ge c\sum_{n\ge n_0}\#\calI_n(\Gamma)k_n\ell_n^\alpha \notag\\
&\ge c\mu(\Gamma)\sum_{n\ge n_0}N_nk_n\ell_n^\alpha \notag\\
&=c\mu(\Gamma)\sum_{n\ge n_0}a_nk_n
=\infty,
\label{eq:energy-diverges}
\end{align}
where we used \eqref{eq:good-arcs-count}, \eqref{eq:N-ell}, and
\eqref{eq:a-k-diverge}.  The divergence in \eqref{eq:energy-diverges}
contradicts \eqref{eq:assume-finite-energy}.  Therefore, by
\eqref{eq:mu-in-Malpha}, the constructed measure belongs to
\(\calM_\alpha(\T)\setminus\Def_m\).  Together with the first inclusion in
\eqref{eq:intro-IN-inclusions}, this proves
\(\Def_m\ne\calM_\alpha(\T)\).
\end{proof}

\begin{remark}
There is no conflict with the positive part of \cite[Theorem~1.3]{IN24}.
That result gives constructibility only for measures concentrated on countable
unions of \(\beta\)-Beurling--Carleson sets with \(\beta<\alpha\). Since the
measure constructed above is not a deficiency measure, it cannot belong to any
of those smaller classes.
\end{remark}

\begin{remark}
The endpoint nature of the construction is reflected in the two competing
series.  The geometric $\alpha$-Beurling--Carleson condition requires
$
\sum_n a_nk_n^\alpha<\infty,
$
whereas the PDE argument produces the divergent series
$
\sum_n a_nk_n=\infty.
$
Because $0<\alpha<1$, one can choose $a_n$ and $k_n$ so that the first series
converges while the second diverges.
\end{remark}


\begin{thebibliography}{99}
\bibitem{Ahe79}
P.~Ahern,
\emph{The mean modulus and the derivative of an inner function},
Indiana Univ. Math. J. \textbf{28} (1979), no.~2, 311--347,
doi: \href{https://doi.org/10.1512/iumj.1979.28.28022}{10.1512/iumj.1979.28.28022}.

\bibitem{AC74}
P.~R. Ahern and D.~N. Clark,
\emph{On inner functions with {$H^p$} derivative},
Michigan Math. J. \textbf{21} (1974), 115--127,
doi: \href{https://doi.org/10.1307/mmj/1029001255}{10.1307/mmj/1029001255}.

\bibitem{AG01}
D.~H. Armitage and S.~J. Gardiner,
\emph{Classical potential theory},
Springer Monographs in Mathematics, Springer, London, 2001,
doi: \href{https://doi.org/10.1007/978-1-4471-0233-5}{10.1007/978-1-4471-0233-5}.

\bibitem{BM92}
C.~Bandle and M.~Marcus,
\emph{``Large'' solutions of semilinear elliptic equations: existence,
uniqueness and asymptotic behaviour},
J. Anal. Math. \textbf{58} (1992), 9--24,
doi: \href{https://doi.org/10.1007/BF02790355}{10.1007/BF02790355}.

\bibitem{BM98}
C.~Bandle and M.~Marcus,
\emph{On second-order effects in the boundary behaviour of large solutions
of semilinear elliptic problems},
Differential Integral Equations \textbf{11} (1998), no.~1, 23--34,
doi: \href{https://doi.org/10.57262/die/1367414131}{10.57262/die/1367414131}.

\bibitem{Beu40}
A.~Beurling,
\emph{Ensembles exceptionnels},
Acta Math. \textbf{72} (1940), 1--13,
doi: \href{https://doi.org/10.1007/BF02546325}{10.1007/BF02546325}.

\bibitem{Bur86}
K.~Burdzy,
\emph{Brownian excursions and minimal thinness, III: Applications to the
angular derivative problem},
Math. Z. \textbf{192} (1986), no.~1, 89--107,
doi: \href{https://doi.org/10.1007/BF01162023}{10.1007/BF01162023}.

\bibitem{Car52}
L.~Carleson,
\emph{Sets of uniqueness for functions regular in the unit circle},
Acta Math. \textbf{87} (1952), 325--345,
doi: \href{https://doi.org/10.1007/BF02392289}{10.1007/BF02392289}.

\bibitem{Cul71}
M.~R. Cullen,
\emph{Derivatives of singular inner functions},
Michigan Math. J. \textbf{18} (1971), no.~3, 283--287,
doi: \href{https://doi.org/10.1307/mmj/1029000692}{10.1307/mmj/1029000692}.

\bibitem{Doo59}
J.~L. Doob,
\emph{A non-probabilistic proof of the relative Fatou theorem},
Ann. Inst. Fourier (Grenoble) \textbf{9} (1959), 293--300,
doi: \href{https://doi.org/10.5802/aif.93}{10.5802/aif.93}.

\bibitem{Doob84}
J.~L. Doob,
\emph{Classical potential theory and its probabilistic counterpart},
Grundlehren der mathematischen Wissenschaften, vol.~262,
Springer-Verlag, New York, 1984,
doi: \href{https://doi.org/10.1007/978-1-4612-5208-5}
{10.1007/978-1-4612-5208-5}.

\bibitem{Dur70}
P.~L. Duren,
\emph{Theory of {$H^p$} spaces},
Pure and Applied Mathematics, vol.~38, Academic Press, New York--London, 1970.

\bibitem{GM05}
J.~B. Garnett and D.~E. Marshall,
\emph{Harmonic measure},
New Mathematical Monographs, vol.~2, Cambridge University Press, Cambridge, 2005,
doi: \href{https://doi.org/10.1017/CBO9780511546617}{10.1017/CBO9780511546617}.

\bibitem{GT01}
D.~Gilbarg and N.~S. Trudinger,
\emph{Elliptic partial differential equations of second order},
2nd ed., Classics in Mathematics, Springer, Berlin, 2001,
doi: \href{https://doi.org/10.1007/978-3-642-61798-0}{10.1007/978-3-642-61798-0}.

\bibitem{Ivrii19}
O.~Ivrii,
\emph{Prescribing inner parts of derivatives of inner functions},
J. Anal. Math. \textbf{139} (2019), no.~2, 495--519,
doi: \href{https://doi.org/10.1007/s11854-019-0064-0}{10.1007/s11854-019-0064-0}.

\bibitem{IN24}
O.~Ivrii and A.~Nicolau,
\emph{Beurling--Carleson sets, inner functions and a semilinear equation},
Anal. PDE \textbf{17} (2024), no.~7, 2585--2618,
doi: \href{https://doi.org/10.2140/apde.2024.17.2585}{10.2140/apde.2024.17.2585}.

\bibitem{Kel57}
J.~B. Keller,
\emph{On solutions of {$\Delta u=f(u)$}},
Comm. Pure Appl. Math. \textbf{10} (1957), 503--510,
doi: \href{https://doi.org/10.1002/cpa.3160100402}{10.1002/cpa.3160100402}.

\bibitem{Mas13}
J.~Mashreghi,
\emph{Derivatives of inner functions},
Fields Inst. Monogr., vol.~31, Springer, New York, 2013,
doi: \href{https://doi.org/10.1007/978-1-4614-5611-7}{10.1007/978-1-4614-5611-7}.

\bibitem{Oss57}
R.~Osserman,
\emph{On the inequality {$\Delta u\ge f(u)$}},
Pacific J. Math. \textbf{7} (1957), no.~4, 1641--1647,
doi: \href{https://doi.org/10.2140/pjm.1957.7.1641}{10.2140/pjm.1957.7.1641}.

\bibitem{Ponce16}
A.~C. Ponce,
\emph{Elliptic PDEs, measures and capacities: From the Poisson equation to
nonlinear Thomas--Fermi problems},
EMS Tracts in Mathematics, vol.~23, European Mathematical Society, Z\"urich, 2016,
doi: \href{https://doi.org/10.4171/140}{10.4171/140}.

\end{thebibliography}
\end{document}